\documentclass[a4paper,          % Papierformat
               12pt,             % Schriftgröße
               bibliography=totoc,         % Literaturverzeichnis im Inhaltsverzeichnis, ebenso Index
               index=totoc,
               notitlepage]{scrartcl}

\usepackage[utf8]{inputenc}
\usepackage{amsmath}
\usepackage{amsthm}
\usepackage{amsfonts}
\usepackage{amssymb}

\usepackage{ifpdf}
\usepackage{url}

\ifpdf
 \usepackage[pdftex]{graphicx}
 \usepackage[pdftex]{color}
\else
 \usepackage{graphicx}
 \usepackage{color}
\fi

\usepackage{verbatim}

\usepackage{enumitem}
\newlist{romanlist}{enumerate}{5}
\setlist[romanlist]{label*=\textup{(}\textit{\roman*}\textup{)}}

\newcommand{\subref}[2]{\mbox{\ref{#1}\hspace{0.4ex}\ref{#2}}}
\newcommand{\ol}{\overline}
\newcommand{\ul}{\underline}

\newcommand{\Mbf}{\mathbf{M}}
\newcommand{\Pbf}{\mathbf{P}}
\newcommand{\Ubf}{\mathbf{U}}

\newcommand{\Rb}{\mathbb{R}}

\newcommand{\Ccal}{\mathcal{C}}
\newcommand{\Ecal}{\mathcal{E}}
\newcommand{\Hcal}{\mathcal{H}}
\newcommand{\Ncal}{\mathcal{N}}
\newcommand{\Pcal}{\mathcal{P}}
\newcommand{\Tcal}{\mathcal{T}}
\newcommand{\Xcal}{\mathcal{X}}
\newcommand{\Ycal}{\mathcal{Y}}
\newcommand{\Zcal}{\mathcal{Z}}

\DeclareMathOperator{\supp}{supp}

\newcommand{\bdKtopA}{\partial\KtopA}

\newcommand{\card}[1]{\ensuremath{{}[#1]{}}}
\newcommand{\cardn}{\card{\numvar}}
\newcommand{\cardN}{\card{N}}
\newcommand{\csA}{\Mbf_{\suffstat}}

\newcommand{\Dbar}{\ol D}
\newcommand{\DbarE}{\ensuremath{\Dbar_{\Expfam}}}
\newcommand{\DE}{\ensuremath{D_{\Expfam}}}
\newcommand{\DNkof}[2]{\ensuremath{D_{#1,#2}}}
\newcommand{\DNk}{\DNkof{N}{k}}

\newcommand{\Expfam}{\ensuremath{\Ecal}}
\newcommand{\expfamuniref}{\ensuremath{\Hcal_{\unimeasure}}}
\newcommand{\exttangspace}{\tilde{\tangspace}}

\newcommand{\ID}[2]{\ensuremath{D(#1\|#2)}}
\newcommand{\IDr}[1]{\ensuremath{\ID{#1}{\refmeasure}}}
\newcommand{\inn}[1]{\ensuremath{{#1}^{\circ}}}

\newcommand{\KtopA}{\Ubf_{\normspace}}

\newcommand{\mommap}{\ensuremath{\pi_{\suffstat}}}

\newcommand{\normspace}{\Ncal}
\newcommand{\numvar}{n}

\newcommand{\pms}{\Pbf} % set of probability measures
\newcommand{\pmsX}{\Pbf(\Xcal)} % set of probability measures
\newcommand{\pospartmap}{\Psi^{+}}
\newcommand{\proj}[2]{\ensuremath{#1_{#2}}}
\newcommand{\projE}[1]{\ensuremath{\proj{#1}{\Expfam}}}
\newcommand{\projPE}{\projE{P}}
\newcommand{\PsiE}{\ensuremath{\Psi_{\Expfam}}}

\newcommand{\refmeasure}{\nu}
\newcommand{\refmeasurex}{\ensuremath{\refmeasure_{x}}}
\newcommand{\Rnn}{\Rb_{\ge}}

\newcommand{\Set}[1]{\ensuremath{\left\{#1\right\}}}
\newcommand{\spms}[1]{\ensuremath{\inn{\pms(#1)}}} % set of strictly positive probability measures
\newcommand{\spmsX}{\spms{\Xcal}}
\newcommand{\suffstat}{A}

\newcommand{\tangspace}{\Tcal}

\newcommand{\unimeasure}{\ensuremath{\mathbf{1}}}

  \newtheorem{lemma}{Lemma}%[section]
  \newtheorem{prop}[lemma]{Proposition}
  \newtheorem{thm}[lemma]{Theorem}
  \newtheorem{cor}[lemma]{Corollary}
  \newtheorem{conj}[lemma]{Conjecture}
  \newtheorem*{conjure*}{Conjecture~\ref{con:DNk-partmod}}
  
\theoremstyle{definition}
  \newtheorem{defi}[lemma]{Definition}

\theoremstyle{remark}
  \newtheorem{rem}[lemma]{Remark}
  \newtheorem{ex}[lemma]{Example}

\author{Johannes~Rauh}
\title{Optimally approximating exponential families}
\date{October 27, 2011}

\begin{document}
\maketitle

\begin{abstract}
  This article studies exponential families $\Expfam$ on finite sets such that the information divergence
  $\ID{P}{\Expfam}$ of an arbitrary probability distribution from $\Expfam$ is bounded by some constant $D>0$.  A
  particular class of low-dimensional exponential families that have low values of $D$ can be obtained from partitions
  of the state space.  The main results concern optimality properties of these partition exponential families.
  Exponential families where
  % The most precise result are in the case
  $D=\log(2)$ are studied in detail.  This case is special, because if $D<\log(2)$, then $\Expfam$ contains all
  probability measures with full support.
\end{abstract}

\section{Introduction}
\label{sec:intro}

Let $\Xcal$ be a finite set of cardinality $N$, and denote by $\pmsX$ the set of probability distributions on $\Xcal$.
The information divergence $\ID PQ$ is a natural distance measure on~$\pmsX$.  For any exponential family $\Expfam$ on
$\Xcal$ (as defined in Section~\ref{sec:prelim}) and any $P\in\pmsX$ write $\DE(P)=\inf_{Q\in\Ecal}D(P\|Q)$.  This
article discusses the following question:
\begin{itemize}
\item Let $D>0$, and choose a partial order on the exponential families.  Which exponential families are minimal among
  all exponential families $\Expfam$ satisfying $\max \DE \le D$?  What is the answer to this question under further
  constraints on~$\Expfam$?
\end{itemize}

This question is related to finding the maximizers of the information divergence from an exponential family, a problem
which was first formulated by Nihat Ay in~\cite{Ay02:Pragmatic_structuring}.  See~\cite{Rauh11:Thesis} for an overview
and further references.  The present work builds on recent progress in~\cite{Rauh11:Finding_Maximizers}
and~\cite{MatusRauh11:Maximization-ISIT2011}.

There are at least two partial orders of interest:
\begin{romanlist}
\item
  \label{en:dim-order}
%  The exponential families may be ordered partially by their dimensions.
  The partial order induced by the dimensions of the exponential families.
\item
  \label{en:inclusion-order}
%  The exponential families may be ordered partially by inclusion.
  The partial order by inclusion.
\end{romanlist}
The partial order~\ref{en:dim-order} is particularly important for applications, since the dimension of an exponential
family is one of the most important invariants that determine the complexity of all computations.  The partial
order~\ref{en:inclusion-order} can be seen as a ``local relaxation'': A candidate exponential family $\Expfam$ is only
compared to ``similar'' exponential families, contained in $\Expfam$.

\begin{defi}
  Let $\Hcal$ be a set of exponential families.  An exponential family $\Expfam\in\Hcal$ is called \emph{inclusion
    $D$-optimal among $\Hcal$} for some $D\ge\max\DE$ if every $\Expfam'\in\Hcal$ strictly contained in $\Expfam$
  satisfies $\max\DE \le D < \max D_{\Expfam'}$.  An exponential family $\Expfam\in\Hcal$ is called \emph{dimension
    $D$-optimal among $\Hcal$} if every exponential family $\Expfam'\in\Hcal$ of smaller dimension satisfies $\max\DE
  \le D < \max D_{\Expfam'}$.  Exponential families that are inclusion or dimension $D$-optimal among $\Hcal$ for some
  $D$ are also called \emph{inclusion} or \emph{dimension optimal among~$\Hcal$}, without reference to $D$.  If $\Hcal$
  equals the set of all exponential families, then the reference to $\Hcal$ may be omitted in all definitions.
  Let
  \begin{equation*}
    \DNk(\Hcal) = \min\Set{\max\DE:\Expfam\in\Hcal\text{ is an exponential family of dimension }k\text{ on }\cardN}.
  \end{equation*}

\end{defi}
As an example, the set $\Hcal$ may be the set of hierarchical models, the set of graphical models or the set
$\expfamuniref$ of exponential families containing the uniform distribution.  Obviously, any dimension optimal
model is also inclusion optimal.  The converse statement does not hold, see Example~\ref{ex:3-optimality} below.

A $D$-optimal exponential family $\Expfam$ can approximate arbitrary probability measures well, up to a maximal
divergence of $D$.  Yaroslav Bulatov proposed to use such exponential families in machine learning (personal
communication), for example when using the \emph{minimax algorithm}~\cite{ZhuWuMumford97:Minimax_Entropy_Principle} by
Zhu, Wu and Mumford or the \emph{feature induction
  algorithm}~\cite{DellaPietra2Lafferty97:Inducing_Features_of_Random_Fields} by Della Pietra, Della Pietra and
Lafferty.  Both algorithms inductively construct an exponential family by adding functions 
(``features'') to the tangent space in order to approximate a given distribution.
Applications of the results of the present paper to machine learning will not be discussed in here, but in a future
work.

% The tangent spaces of optimal exponential families may be natural candidates for features, if no or little expert
% knowledge is available.  Some suggestions for the choice of possible candidate features will be given in Section~....

One motivation to restrict the class $\Hcal$ of exponential families is that the learning system may not be able to
represent arbitrary exponential families.  Another motivation is given by Jaynes' principle of maximum
entropy~\cite{Jaynes57:Information_Theory_and_Statistical_Mechanics}, which suggests to use the class $\expfamuniref$ of
exponential families with uniform reference measure.

This paper also introduces the class of \emph{partition models} (see Section~\ref{sec:partmod}): A probability measure
$P$ belongs to the partition model associated to a partition $\Xcal'=(\Xcal^{1},\dots,\Xcal^{N'})$ if the restriction of
$P$ to each block $\Xcal^{i}$ is uniform.  Conjecture~\ref{con:DNk-partmod} relates partition models to the above
question:
\begin{conjure*}
  $\DNk = \log\lceil\frac{N}{k+1}\rceil$, and the dimension $\DNk$-optimal exponential families containing the uniform
  distribution are partition models.
\end{conjure*}
The results in Section~\ref{sec:maxDdE=log2} show that the conjecture is true if $\lceil\frac{N}{k+1}\rceil\le 2$, and
Theorem~\ref{thm:logNk-optimality} proves the conjecture if $k+1$ divides $N$.

This paper is organized as follows: Section~\ref{sec:prelim} collects the necessary preliminaries about exponential
families and the information divergence.  Section~\ref{sec:partmod} introduces partition models and studies their basic
properties.  $\log(2)$-optimal exponential families $\Expfam$ are studied in Section~\ref{sec:maxDdE=log2}.
Section~\ref{sec:opti} presents results on $D$-optimal exponential families for arbitrary $D$.

\section{Preliminaries}
\label{sec:prelim}

This section collects known facts that are needed in later sections.  It starts with some notions from matroid theory
before defining exponential families, the information divergence and hierarchical models.  The last part discusses the
function $\DbarE$, which arises naturally when studying the maximizers of $\DE$.

\subsection{Circuits}
\label{ssec:prel:matroids}

This section recalls some elementary notions from the theory of matroids.  Only representable matroids will play a role,
but nevertheless the language of abstract matroids is useful. See~\cite{Oxley92:Matroid_Theory} for an introduction.

\begin{defi}
  Let $\normspace$ be a linear subspace of $\Rb^{\Xcal}$.  The \emph{support} of $u\in\normspace$ is defined as
  $\supp(u):=\{x\in\Xcal:u(x)\neq 0\}$.  A vector $v\in\normspace\setminus\{0\}$ is called a \emph{circuit vector} if
  and only if for any $u\in\normspace$ satisfying $\supp(u)\subseteq\supp(v)$ there exists $\alpha\in\Rb$ such that $u =
  \alpha v$.  In other words, circuit vectors are vectors with minimal support.  The support $\supp(u)$ of a circuit
  vector $u$ is called a \emph{circuit}.
  A finite set $\Ccal\subseteq\normspace$ is a \emph{circuit basis} if and only if the map $u\in\Ccal\mapsto\supp(u)$ is
  injective and maps onto the set of circuits.
\end{defi}
\begin{lemma}
  \label{lem:circuitinvector}
  For every nonzero vector $u\in\normspace$ and any $x\in\Xcal$ such that $u(x)\neq 0$ there exists a
  circuit vector $c\in\normspace$ such that $\supp(c)\subseteq\supp(u)$ and $c(x)\neq 0$.
\end{lemma}
\begin{proof}
  Let $c$ be a vector with inclusion-minimal support that
  satisfies $\supp(c)\subseteq\supp(u)$ and $c(x)\neq 0$.  If $c$ is not a circuit vector, then there exists a circuit
  vector $c'$ with $\supp(c')\subset\supp(c)$.  A suitable linear combination $c + \alpha c'$, $\alpha\in\mathbb{R}$
  gives a contradiction to the minimality of $c$.
\end{proof}
It follows that any circuit basis of $\normspace$ contains a spanning set.

\subsection{Exponential families and the information divergence}
\label{ssec:prel:expfams}

In this work only exponential families on a finite set $\Xcal$ are studied, for the information divergence from a
finite-dimensional exponential family on an infinite set is usually unbounded, cf.~Theorem~\ref{thm:logNk-optimality}.
See~\cite{Brown86:Fundamentals_of_Exponential_Families} and~\cite{CsiszarShields04:Information_Theory_and_Statistics}
for an introduction to exponential families and the information divergence.

Let $\exttangspace$ be a linear subspace of $\Rb^{\Xcal}$ containing the constant function, and let $\refmeasure$ be a
strictly positive measure on $\Xcal$.  The set $\Expfam=\Expfam_{\nu,\exttangspace}$ of all probability measures on
$\Xcal$ of the form
\begin{equation}
  \label{eq:expfam}
  P_{\vartheta}(x) = \frac{\refmeasure(x)}{Z_{\vartheta}} e^{\vartheta(x)}
\end{equation}
is called an \emph{exponential family}.  $\refmeasure$ is a \emph{reference measure}, and $\exttangspace$ will be called
the \emph{extended tangent space} of $\Expfam$.  The extended tangent space carries its name since its image modulo the
constant functions is isomorphic to the tangent space of the manifold $\Expfam$ at any point.  The orthogonal complement
$\normspace := \exttangspace^{\perp}$ will be called the \emph{normal space} of $\Expfam$.  The normal space is
orthogonal to the tangent space of $\Expfam$ at any point $P\in\Expfam$ with respect to the Fisher metric at $P$.  The
topological closure of $\Ecal$ will be denoted by $\ol \Ecal$.

The exponential family $\Expfam_{\nu,\exttangspace}$ can be parametrized as follows:
If $a_{1},\dots,a_{h}\in\Rb^{\Xcal}$ form a spanning set of $\exttangspace$, then $\Expfam$ consists of all probability
distributions of the form
\begin{equation}
  \label{eq:expfam-A}
  P_{\theta}(x) = \frac{\refmeasure(x)}{Z_{\theta}} \exp\left(\sum_{i=1}^{h}\theta_{i}a_{i}(x)\right).
\end{equation}
In this formula $\theta\in\Rb^{h}$ is a vector of parameters and $Z_{\theta}$ ensures normalization.  The matrix $A =
(a_{i}(x))_{i,x}\in\Rb^{h\times\Xcal}$ is called a \emph{sufficient statistics} of~$\Ecal$.  The linear map
corresponding to $A$ is called the \emph{moment map}, denoted by $\mommap$.  The columns of $A$ will be denoted by
$A_{x},x\in\Xcal$.  The normal space of $\Ecal$ equals $\normspace=\{u\in\ker A:\sum_{x}u(x)=0\}$.
The convex hull of $\{A_{x}:x\in\Xcal\}$ is a polytope called the \emph{convex support} $\csA$ of $\Expfam$.  This
polytope is independent of the choice of $\suffstat$ up to an affine transformation.

Any function $u\in\Rb^{\Xcal}$ can be decomposed uniquely as a difference $u=u^{+}-u^{-}$ of non-negative functions such
that $\supp(u^{+})\cap\supp(u^{-}) = \emptyset$.  The following implicit description of an exponential family is useful
in many contexts.
\begin{thm}
  \label{thm:impl-circuit-eqs}
  Let $\Expfam$ be an exponential family with normal space $\normspace$ and reference measure~$\refmeasure$, and let
  $\Ccal$ be a circuit basis of~$\normspace$.  A probability measure $P$ on $\Xcal$ belongs to $\ol\Expfam$ if and only
  if $P$ satisfies
  \begin{equation}
    \label{eq:impl-circuit-eqs}
    \prod_{x\in\Xcal}\left(\frac{P(x)}{\nu(x)}\right)^{u^{+}(x)} = \prod_{x\in\Xcal}\left(\frac{P(x)}{\nu(x)}\right)^{u^{-}(x)}, \qquad \text{ for all }  u=u^{+}-u^{-} \in \Ccal.
  \end{equation}
\end{thm}
\begin{proof}
  See \cite[Theorem 10]{RKA10:Support_Sets_and_Or_Mat}.
\end{proof}
  Let $\Expfam_{1},\dots,\Expfam_{c}\subseteq\pmsX$.
  The \emph{mixture} of $\Expfam_{1},\dots,\Expfam_{c}$ is the set of probability measures
  \begin{equation*}
    \Set{
      P = \sum_{i=1}^{c}\lambda_{i}P_{i} : P_{1}\in\Expfam_{1},\dots,P_{c}\in\Expfam_{c}\text{ and }
      \lambda\in\Rnn^{c},\sum_{i=1}^{c}\lambda_{i}=1
    }.
  \end{equation*}
\begin{cor}
  \label{cor:disconnected-matroid}
  Let $\Expfam$ be an exponential family with normal space $\normspace$.  Let $\Ycal\subset\Xcal$.  If every circuit
  vector $c\in\normspace$ satisfies $\supp(c)\subseteq\Ycal$ or $\supp(c)\subseteq\Xcal\setminus\Ycal$, then
  $\ol\Expfam$ equals the mixture of $\ol\Expfam\cap\pms(\Ycal)$ and $\ol\Expfam\cap\pms(\Xcal\setminus\Ycal)$.
\end{cor}
\begin{proof}
  For any probability measure $P$ and subset $\Ycal\subseteq\Xcal$ define the \emph{truncation}
  $P^{\Ycal}$ as follows: If $P(\Ycal)>0$, then
  \begin{equation*}
    P^{\Ycal}(x) =
    \begin{cases}
      \frac{1}{P(\Ycal)}P(x), &\text{ if }x\in\Ycal,\\
      0, &\text{ else};
    \end{cases}
  \end{equation*}
  otherwise let $P^{\Ycal}$ be an arbitrary probability distribution on $\Ycal$.  By Theorem~\ref{thm:impl-circuit-eqs},
  a probability measure $P\in\pmsX$ with full support lies in $\Expfam$ if and only if its truncations $P^{\Ycal}$ and
  $P^{\Xcal\setminus\Ycal}$ lie in $\Expfam\cap\pms(\Ycal)$ and $\Expfam\cap\pms(\Xcal\setminus\Ycal)$, respectively.
\end{proof}
The corollary can be reformulated as follows, using terminology from matroid theory: If
$\Xcal_{1},\dots,\Xcal_{c}$ are the connected components of the
matroid of $\normspace$, then $\ol\Expfam$ equals the mixture of $\ol{\Expfam_{1}},\dots,\ol{\Expfam_{c}}$, where
$\Expfam_{i}=\ol\Expfam\cap\spms{\Xcal_{i}}$ is an exponential family on $\Xcal_{i}$ for $i=1,\dots,c$.

\medskip
The \emph{information divergence} (also known as the \emph{Kullback-Leibler divergence} or \emph{relative entropy}) of
positive measures $P$, $Q$ is defined as
\begin{equation}
  \label{eq:def:infdiv}
  D(P\|Q) = \sum_{x\in\Xcal}P(x)\log\left(\frac{P(x)}{Q(x)}\right).
\end{equation}
with the convention that $0\log 0 = 0\log(0/0) = 0$.  It is finite unless $\supp(P)$ is not contained in $\supp(Q)$.
If $\refmeasure$ equals the counting measure on $\Xcal$ (i.e.~$\refmeasurex=1$ for all $x$), then $\ID{P}\refmeasure$
equals minus the Shannon entropy $H(P)$.  If $P$ and $Q$ are probability measures, then $\ID PQ$ is strictly positive
unless $P=Q$.

Let $\Ecal$ be an exponential family.  For any probability measure $P$ on $\Xcal$ there is a unique probability
distribution $\projPE\in\ol\Ecal$ such that $D(P\|\projPE) = \inf_{Q\in\Ecal}D(P\|Q)$, see
\cite{CsiszarMatus08:GMLE_for_Exp_Fam}.  The measure $\projPE$ is called the \emph{(generalized) $rI$-projection} of $P$
to $\Expfam$ or the \emph{(generalized) MLE}.  It can also be characterized as the unique probability measure
$\projPE\in\ol\Ecal$ such that $P - \projPE\in\normspace$.  Alternatively, $\projPE$ minimizes the function $\ID
Q\refmeasure$ on $\{Q\in\Pcal(\Xcal) : P-Q\in\normspace\}$.  In particular, if $\refmeasure$ is the counting measure,
then $\projPE$ maximizes the entropy.

\subsection{Hierarchical loglinear models}
\label{ssec:prel:hiermod}

Let $\Xcal_{1},\dots,\Xcal_{\numvar}$ be finite sets of cardinality $|\Xcal_{i}|=N_{i}$, and let
$\Xcal=\Xcal_{1}\times\dots\times\Xcal_{\numvar}$.  For any subset $S\subseteq\cardn$ let $\Xcal_{S}=\times_{i\in
  S}\Xcal_{i}$.
The restrictions $X_{i}:\Xcal\to\Xcal_{i}$ to the subsystems can be viewed as random variables, and hierarchical models
can be used to study the relationship of these discrete random variables.  This section summarizes the main facts which
are needed in the following.  See~\cite{Lauritzen96:Graphical_Models}
and~\cite{DrtonSturmfelsSullivant09:Algebraic_Statistics} for further information.

\begin{defi}
  For any family $\Delta$ of subsets of $\cardn$ let
  $\Expfam'_{\Delta}$ be the set of all probability measures $P\in\spmsX$ that can be written in the form
  \begin{equation}
    \label{eq:hierarchical-product}
    P(x) = \prod_{S\in\Delta} f_{S}(x),
  \end{equation}
  where each $f_{S}$ is a non-negative function on $\Xcal$ that depends only on those components of $x$ lying in $S$.
  In other words, $f_{S}(x) = f_{S}(y)$ for all $x=(x_{i})_{i=1}^{\numvar},y=(y_{i})_{i=1}^{\numvar}\in\Xcal$ satisfying
  $x_{i}=y_{i}$ for all $i\in S$.
  The \emph{hierarchical exponential family}
  $\Expfam_{\Delta}$ of $\Delta$
  with parameters $N_{1}$, $N_{2}$, \dots, $N_{\numvar}$ is defined as $\Expfam'_{\Delta}\cap\spmsX$.  The closure of
  $\Expfam_{\Delta}$ (which equals the closure of $\Expfam'_{\Delta}$) is called the \emph{hierarchical
    model} of $\Delta$ with parameters $N_{1}$, $N_{2}$, \dots, $N_{\numvar}$.
\end{defi}
At first sight one might think that $\Expfam'_{\Delta}=\ol{\Expfam_{\Delta}}$.  Unfortunately, this is not true,
see~\cite{GeigerMeekSturmfels06:Toric_Algebra_Graphical_Models}.  For certain applications, when the factorizability
probability is important, one might want to call $\Expfam'_{\Delta}$ a hierarchical model.  When studying
optimization problems it is more important that the models are closed.

For any $S\subseteq\{1,\dots,\numvar\}$ the subset of $\Rb^{\Xcal}$ of functions that only depend on the $S$-components
can be naturally identified with $\Rb^{\Xcal_{S}}$.  The projection $\Xcal\to\Xcal_{S}$ induces a natural injection
$\Rb^{\Xcal_{S}}\to\Rb^{\Xcal}$.

It is easy to see that hierarchical exponential families are indeed exponential families: Namely,
\eqref{eq:hierarchical-product} implies that $\Expfam_{\Delta}$ consists of all
$P\in\spmsX$ that satisfy
\begin{equation*}
  \left(\log(P(x))\right)_{x\in\Xcal} \in \sum_{S\in\Delta}\Rb^{\Xcal_{S}}\subseteq\Rb^{\Xcal}.
\end{equation*}
Therefore, $\Expfam_{\Delta}$ is an exponential family with uniform reference measure and extended tangent space
$\tilde\tangspace = \sum_{S\in\Delta}\Rb^{\Xcal_{S}}$.  This vector space sum is not direct, since every summand
contains $\unimeasure$.  There is a natural sufficient statistics: The marginalization maps
$\pi_{S}:\Rb^{\Xcal}\mapsto\Rb^{\Xcal_{S}}$ defined for $S\subseteq\{1,\dots,n\}$ via
\begin{equation*}
  \pi_{S}(v) (x) = \sum_{y\in\Xcal: y_{i}=x_{i}\text{ for all }i\in S} v(y)
\end{equation*}
induce the moment map
\begin{equation*}
  \pi_{\Delta}: v\in\Rb^{\Xcal}\mapsto (\pi_S(v))_{S\in\Delta}\in\bigoplus_{S\in\Delta}\Rb^{\Xcal_{S}},
\end{equation*}
where $\oplus$ denotes the (external) direct sum of vector spaces.
\begin{lemma}
  \label{lem:csA-of-hierarchical}
  Let $\Delta$ be a collection of subsets of $\cardn$, and let $K=\cup_{J\in\Delta}J$.  The marginal polytope of
  $\Delta$ is (affinely equivalent to) a 0-1-polytope with $\prod_{i\in K}N_{i}$
  vertices.
\end{lemma}
\begin{proof}
  The moment map $\pi_{\Delta}$ corresponds to a sufficient statistics $\suffstat_{\Delta}$ that only has entries $0$
  and $1$, so $\csA$ is a 0-1-polytope.  The set of vertices of $\csA$ is a subset of $\{\suffstat_{x}:x\in\Xcal\}$.
  Let $x=(x_{i})_{i=1}^{\numvar},y=(y_{i})_{i=1}^{\numvar}\in\Xcal$.  If $x_{i}=y_{i}$ for all $i\in K$, then
  $\suffstat_{x}=\suffstat_{y}$, so $\csA$ has at most $\prod_{i\in K}N_{i}$ vertices.  If $x_{i}\neq y_{i}$ for some
  $i\in K$, then $\suffstat_{x}\neq\suffstat_{y}$, so the set $\{\suffstat_{x}:x\in\Xcal\}$ has cardinality
  $\prod_{i\in K}N_{i}$.  Since this set consists of 0-1-vectors and since no 0-1-vector is a convex combination of
  other 0-1-vectors, it follows that the set of vertices of $\csA$ equals $\{\suffstat_{x}:x\in\Xcal\}$ and has
  cardinality $\prod_{i\in K}N_{i}$.
\end{proof}

\subsection[The function D-bar]{The function $\DbarE$}
\label{ssec:prel:DbarE}

The function $\DE$ is related to the function
\begin{equation*}
  \DbarE(u) = \sum_{x\in X}u(x)\log\frac{|u(x)|}{\refmeasurex}
\end{equation*}
defined on $\normspace$~\cite{Rauh11:Finding_Maximizers}.  The function $\DbarE$ satisfies $\DbarE(\alpha u) =
\alpha\DbarE$ for all $\alpha\in\Rb$ and $u\in\normspace$.  It will mostly be considered on a subset $\bdKtopA$ of
$\normspace$, defined as follows:

\begin{defi}
  For any $v\in\Rb^{\Xcal}$ and $\Zcal\subseteq\Xcal$ write $v(\Zcal):=\sum_{x\in\Zcal}v(x)$.  Let
  \begin{equation*}
    \bdKtopA:=\left\{ u\in\normspace : u^{+}(\Xcal)= u^{-}(\Xcal)= 1 \right\}.
  \end{equation*}
  The map $\pospartmap:u\mapsto u^{+}$ maps $\bdKtopA$ to a subset of $\pmsX$.  A probability distribution in the image
  of $\pospartmap$ is called a \emph{kernel distribution}.

  In the other direction there is the natural map $\PsiE:\pmsX\setminus\ol\Expfam\to\normspace$, defined via
  \begin{equation*}
    \PsiE(P) = \frac{P - \projPE}{(P - \projPE)^{+}(\Xcal)}.
  \end{equation*}
  The denominator makes sure that the image of $\PsiE$ lies in $\bdKtopA$.  Since $P=\projPE$ if and only if
  $P\in\ol\Expfam$, the map is well-defined on $\pmsX\setminus\ol\Expfam$.
\end{defi}
\begin{thm}
  \label{thm:Dualmaxi-glob}
  Let $\Expfam$ be an exponential family with normal space $\normspace\neq 0$.  The map $\PsiE$ restricts to a bijection
  from the set of local maximizers of $\DE$ to the set of local maximizers of $\DbarE$.  An inverse is given by the
  restriction of the map $\pospartmap:u\mapsto u^{+}$.  If $P\in\pmsX$ and $u\in\bdKtopA$ are local maximizers of $\DE$
  and $\DbarE$, respectively, then
  \begin{equation*}
    \DE(P) = \log(1 + \exp(\DbarE(\PsiE(P))))\text{ and }\DE(u^{+}) = \log(1 + \exp(\DbarE(u))).
  \end{equation*}
\end{thm}
\begin{proof}
  See~\cite[Theorem~1]{MatusRauh11:Maximization-ISIT2011}.
\end{proof}
See~\cite{MatusRauh11:Maximization-ISIT2011} and~\cite{Rauh11:Thesis} for further relations between the functions $\DE$ and $\DbarE$.
\begin{cor}
  \label{cor:max-ge-log2}
  Let $\Expfam$ be an exponential family.  If $\ol\Expfam\neq\pmsX$, then $\max\DE\ge\log(2)$.
\end{cor}
\begin{proof}
  Let $u\in\bdKtopA$ be a global maximizer of $\DbarE$.  Since $\DbarE(-u)=-\DbarE(u)$ the maximal value $\DbarE(u)$ is
  non-negative.  Hence $\DE(u^{+}) = \log(1 + \exp(\DbarE(u))) \ge \log(2)$.
\end{proof}
It is straightforward to compute the first-order criticality conditions of $\DbarE$:
\begin{prop}
  \label{prop:critpts-DbarE}
  Let $\Expfam$ be an exponential family with normal space $\normspace$, let $u\in\bdKtopA$ be a local maximizer of
  $\DbarE$, and let $\Ycal=\supp(u)$.  The following statements hold:
  \begin{romanlist}
  \item 
    \label{en:critpt-DbarE:v(u=0)=0}
    $v(\Ycal) = 0$
    for all $v\in\normspace$.
  % \item
  %   \label{en:critpt-DbarE:criteqs}
  %   If $v\in\normspace$ satisfies $\supp (v) \subseteq \Ycal$ and $v(u>0)=0$, then
  %   \begin{equation}
  %     \label{eq:critpt-DbarE:criteqs}
  %     \sum_{x:u\neq 0} v(x) \log \frac{|u(x)|}{\refmeasurex} = 0.
  %   \end{equation}
  \item 
    \label{en:critpt-DbarE:critineqs}
    Let $\projPE$ be the $rI$-projection of $u^{+}$ and $u^{-}$, and
    let $v\in\Ncal$.  Then
    \begin{equation}
      \label{eq:critpt-DbarE:critineqs}
      \sum_{x\in\Xcal\setminus\Ycal}v(x)\log\frac{|v(x)|}{\refmeasurex} \le v^{+}(\Zcal')\DbarE(v_{0}).
    \end{equation}
  \end{romanlist}
\end{prop}
\begin{proof}
  See~\cite{MatusRauh11:Maximization-ISIT2011} or~\cite[Proposition 3.21]{Rauh11:Thesis}.
\end{proof}

\section{Partition models}
\label{sec:partmod}

Partition exponential families are convex exponential families.  The information
divergence from convex exponential families has been studied
in~\cite{MatusAy03:On_Maximization_of_the_Information_Divergence}.  Apart from this, partition exponential families do
not seem to have been studied before, despite their peculiar properties.  In other contexts the name ``partition model''
is used for other mathematical objects, but there seems to be little danger of confusion.

\begin{defi}
  A \emph{partition} $\Xcal'$ of $\Xcal$ is a family
  $\Xcal'=\left\{\Xcal^{1}, \Xcal^{2}, \dots, \Xcal^{N'}\right\}$ of nonempty subsets
  $\Xcal^{i}\subset\Xcal$ such that $\Xcal=\Xcal^{1}\cup\Xcal^{2}\cup\dots\cup\Xcal^{N'}$
  and $\Xcal^{i}\cap\Xcal^{j}=\emptyset$ for all $1\le i<j\le N'$.  The subsets
  $\Xcal^{i}\subseteq\Xcal$ are called the \emph{blocks} of the partition $\Xcal'$.  For any $x\in\Xcal$
  the block $\Xcal^{i}$ containing $x$ is denoted $\Xcal^{x}$.

  The \emph{coarseness} $c(\Xcal')$ of a partition $\Xcal'$ is the cardinality of the largest
  block of $\Xcal'$.  A partition $\Xcal'$ is called \emph{homogeneous} if all
  blocks of $\Xcal'$ have the same cardinality~$c(\Xcal')$.  Partitions are in bijection with equivalence relations, the
  blocks of a partition corresponding to the equivalence classes.  The equivalence relation induced by the partition
  $\Xcal'$ is denoted $\sim_{\Xcal'}$.  In other words $x,y\in\Xcal$ satisfy $x\sim_{\Xcal'}y$ if and only if $x$ and
  $y$ lie in the same block of $\Xcal'$.
\end{defi}

\begin{defi}
  Let $\Xcal'$ be a partition of $\Xcal$.  Denote $\Rb^{\Xcal'}$ the set of functions $\vartheta:\Xcal\to\Rb$ such that
  $x\sim_{\Xcal'}y$ implies $\vartheta(x)=\vartheta(y)$.  The exponential family $\Expfam_{\Xcal'}$ with uniform
  reference measure and extended tangent space $\Rb^{\Xcal'}$ is called the \emph{partition exponential family}
  of~$\Xcal'$, and $\ol{\Expfam_{\Xcal'}}$ is the \emph{partition model} of $\Xcal'$.
\end{defi}
Partition models are, in fact, also linear families: $\ol{\Expfam_{\Xcal'}}$ equals the intersection of $\pmsX$ with the
linear space $\Rb^{\Xcal'}$.  In particular, partition exponential families are convex exponential families.  Convex
exponential families have been studied by Ay and Matúš
in~\cite{MatusAy03:On_Maximization_of_the_Information_Divergence}, which contains more detailed arguments for the
following calculations.  It follows from~\cite[Proposition~1]{MatusAy03:On_Maximization_of_the_Information_Divergence}
that a convex exponential family is a partition exponential family if and only if it contains the uniform distribution.

\begin{rem}
%  \label{rem:part-mod-symmetry}
  Partition models can be used to model symmetries.  This was first noted by Juríček, who used this idea to compute the
  global maximizers of $\DE$ for the multinomial models~\cite{Juricek10:Maximization_from_multinomial_distributions}.
  If a symmetry group $G$ acts on $\Xcal$, then it induces a partition $\Xcal^{G}$ of $\Xcal$ into orbits
  $\Xcal^{1},\dots,\Xcal^{N'}$.  The action of $G$ extends naturally to an action on $\Rb^{\Xcal}$.  Any
  exponential family that consists of $G$-invariant probability measures is a subfamily of
  $\Expfam_{\Xcal^{G}}$ (such exponential families are called
  \emph{$G$-exchangeable}
  in~\cite{Juricek10:Maximization_from_multinomial_distributions}).  Conversely, an arbitrary partition model
  $\ol{\Expfam_{\Xcal'}}$ arises in this way from the group of all permutations $g$ of $\Xcal$ such that
  $g(\Xcal^{i})=\Xcal^{i}$ for all $\Xcal^{i}\in\Xcal'$.
\end{rem}

\begin{lemma}
  \label{lem:cs-simplex-partmodel}
  An exponential family with uniform reference measure and sufficient statistics $\suffstat\in\Rb^{h\times\Xcal}$
  is a partition exponential family if and only if its convex support is a simplex with vertex set
  $\{\suffstat_{x}:x\in\Xcal\}$.
\end{lemma}
\begin{proof}
  A sufficient statistics of $\Expfam_{\Xcal'}$ is given by the characteristic functions $a_{i}=\mathbf{1}_{\Xcal^{i}}$
  of the blocks of $\Xcal'$.
  Any column of $A=(a_{i}(x))_{i,x}$ is a unit vector, and therefore the convex support is a simplex.

  In the other direction
  define an equivalence relation $\sim$ on $\Xcal$ via $x\sim y$ if and only if $\suffstat_{x}=\suffstat_{y}$.  Then
  $\Expfam$ agrees with the partition exponential family of this equivalence relation.
\end{proof}

For partition models the mapping $P\mapsto \projPE$ is easy to compute: 
The equation $AP=A\projPE$ translates into %
$P(\Xcal^{i})=\projPE(\Xcal^{i})$ for $i=1,\dots,N'$.  Therefore,
\begin{equation}
  \label{eq:partition-PE}
  \projPE(x) = \projPE^{\Xcal^{x}}(x) P(\Xcal^{x}),\qquad\text{for all }x\in\Xcal,
\end{equation}
where $\projPE^{\Xcal^{x}}$ denotes the truncation of $\projPE$ to $\Xcal^{x}$.
Since $\projPE$ maximizes the entropy subject to~\eqref{eq:partition-PE}, it follows that $\projPE^{\Xcal^{x}} =
\frac{1}{|\Xcal^{x}|}\unimeasure_{\Xcal^{x}}$ is the uniform distribution on $\Xcal^{x}$.  Hence the $rI$-projection map
$P\mapsto\projPE$ averages over the blocks of the partition.
It follows that
\begin{equation*}
  \DE(P) =
  \sum_{i=1}^{N'} P(\Xcal^{i}) D(P^{\Xcal^{i}}\|\frac{1}{|\Xcal^{i}|}\unimeasure_{\Xcal^{i}})
  = \sum_{i=1}^{N'} P(\Xcal^{i}) \left(\log|\Xcal^{i}| - H(P^{\Xcal^{i}})\right).
\end{equation*}
As a consequence:
\begin{lemma}
  \label{lem:partition-maxDdE}
  If $\ol\Expfam$ is a partition model of a partition $\Xcal^{1},\dots,\Xcal^{N'}$ of coarseness $c$, then
  $\max\DE = \log(c)$.  A probability measure $P\in\pmsX$ maximizes $\DE$ if and only if the following two conditions
  are satisfied:
  \begin{romanlist}
  \item $P(\Xcal^{i})>0$ only if $|\Xcal^{i}|=c$.
  \item $P^{\Xcal^{i}}$ is a point measure for all $i$ such that $|\Xcal^{i}|=c$ and $P(\Xcal^{i})>0$.
  \end{romanlist}
\end{lemma}
\begin{cor}
  \label{cor:part-mod:any-Q-is-rI-of-glob-max}
  Let $\ol\Expfam$ be the partition model of a partition $\Xcal'$ of coarseness $c$, and let $\Zcal$ be the union
  of the blocks of $\Xcal'$ of cardinality $c$.  Then any $Q\in\ol\Expfam$ with support contained in $\Zcal$ is the
  $rI$-projection of some global maximizer of~$\DE$.
  In particular, if $\Xcal'$ is homogeneous, then any $Q\in\ol\Expfam$ is the $rI$-projection of some global
  maximizer of~$\DE$.
\end{cor}
\begin{proof}
  For any $\Xcal^{i}\in\Xcal'$ of cardinality $c$ choose a representative $x_{i}\in\Xcal^{i}$.  Define
  $P\in\pmsX$ by $P(\Xcal^{i})=Q(\Xcal^{i})$ and $P^{\Xcal^{i}}=\delta_{x_{i}}$ for all $i$ such that
  $|\Xcal^{i}|=c$.  Then $\projPE=Q$, so the statement follows from Lemma~\ref{lem:partition-maxDdE}.
\end{proof}

\begin{rem}
  \label{rem:hier-partmod}
  Composite systems have natural homogeneous partitions, which lead to hierarchical models as defined in
  Section~\ref{sec:prelim}: Suppose that $\Xcal=\Xcal_{1}\times\dots\times\Xcal_{\numvar}$ and let $K\subseteq
  \left\{1,\dots,\numvar\right\}$.  Then $K$ induces an equivalence $\sim_{K}$ on $\Xcal$ via $x\sim_{K}y$ if and only
  if $x_{i}=y_{i}$ for all $i\in K$.  The equivalence classes of $\sim_{K}$ form a homogeneous partition $\Xcal^{K}$ of
  $\Xcal$ of coarseness $\prod_{i: i\notin K}N_{i}$.  The corresponding partition model $\ol{\Expfam_{K}}$ consists of
  those probability distributions $P$ satisfying $P(x) = P(y)$ whenever \mbox{$x\sim_{K}y$}.  Therefore, $\Expfam_{K}$
  equals the hierarchical exponential family $\Expfam_{\{K\}}$.
  Conversely, any homogeneous partition $\Xcal'$ can be used to find a bijection of
  $\Xcal$ with a composite system $\Xcal_{1}\times\Xcal_{2}$, where $\Xcal_{1}=\Xcal'$ and
  $\Xcal_{2}\in\Xcal'$.  Then the partition $\Xcal'$ arises from $\sim_{K}$, where $K=\{1\}$.
\end{rem}

\section[Exponential families where the maximal information divergence is log(2)]{Exponential families with $\max\DE=\log(2)$}
\label{sec:maxDdE=log2}

By Corollary~\ref{cor:max-ge-log2} the maximal value of $\DE$ is at least $\log(2)$ unless
$\ol\Expfam=\pmsX$.  This section studies exponential families $\Expfam$ where $\max\DE=\log(2)$.
% If $\max\DE=\log(2)$, then
For such an exponential family, any kernel distribution is a local maximizer of $\DE$.  Furthermore, $\DbarE(u)=0$ for
all $u\in\normspace$ (even if $u\notin\bdKtopA$).
The main results are:
\begin{thm}
  \label{thm:maxDdE=log2}
  Let $\Expfam$ be an exponential family on a finite set $\Xcal$ of cardinality $N$. If $\max\DE=\log(2)$, then the
  dimension of $\Expfam$ is at least $\lceil\frac N2\rceil-1$.
\end{thm}
\begin{thm}
  \label{thm:maxDdE=log2-partition-case}
  Let $\Xcal$ be a finite set of cardinality $N$, and let $\Expfam$ be an exponential family on $\Xcal$ of dimension
  $\lceil\frac N2\rceil-1$ satisfying $\max\DE=\log(2)$.  If $N$ is even, then $\Expfam$ is a partition model.  If $N$
  is odd, then there is a set $\Zcal\subseteq\Xcal$ of cardinality three, a partition model
  $\ol{\Expfam_{\Xcal\setminus\Zcal}}$ on $\Xcal\setminus\Zcal$ and a one-dimensional exponential family
  $\Expfam_{\Zcal}$ on $\Zcal$ such that $\max\ID{\cdot}{\Expfam_{\Xcal\setminus\Zcal}} = \log(2) =
  \max\ID{\cdot}{\Expfam_{\Zcal}}$, and %after an appropriate permutation of $\Xcal$
  the closure $\ol\Expfam$ equals the mixture of $\ol{\Expfam_{\Xcal\setminus\Zcal}}$ and $\ol{\Expfam_{\Zcal}}$.  If
  $\Expfam$ contains the uniform distribution, then $\ol\Expfam$ is a partition model.
\end{thm}
\begin{prop}
  \label{prop:maxDdE=log2,N=3}
  Let $\Xcal=\{1,2,3\}$.  For any $u\in\Rb^{\Xcal}$ such that $u_{1}+u_{2}+u_{3}=0$ there exists a unique
  exponential family $\Expfam$ on $\Xcal$ with normal space $\normspace=\Rb u$ such that $\max\DE=\log(2)$.
\end{prop}
The proofs of the three results will be given below after a series of preliminary lemmas.  Under the additional
assumptions that $N$ is even % and that $\Expfam$ contains the uniform distribution
Theorem~\ref{thm:maxDdE=log2-partition-case} has a simpler proof, see Theorem~\ref{thm:logNk-optimality}.

Let $\Expfam$ be an exponential family with sufficient statistics $\suffstat$ and normal space $\normspace$.
\begin{lemma}
  \label{lem:equiv-classes-facial}
  For any $v_{0},v_{1},\dots,v_{s}\in\normspace$ let $\Zcal=\supp(v_{0})\setminus\cup_{j=1}^{s}\supp(v_{j})$.  Suppose
  that $\max\DE=\log(2)$.  Then
  \begin{equation*}
    \sum_{x\in\Zcal}v(x)\log\frac{|v(x)|}{\refmeasurex} = 0\quad\text{ and }\quad\sum_{x\in\Zcal}v(x)=0\quad\text{ for all }v\in\normspace.
  \end{equation*}
\end{lemma}
\begin{proof}
  The proof is by induction on $s$.  Let $s=0$.  Any $v_{0}\in\normspace$ satisfies $\DbarE(v_{0})=0$ and is a local
  maximizer of~$\DbarE$.  The equality $v(\Zcal)=0$ for all $v\in\normspace$ follows from
  Proposition~\subref{prop:critpts-DbarE}{en:critpt-DbarE:v(u=0)=0}.  Let $\Zcal'=\Xcal\setminus\Zcal$.
  Proposition~\subref{prop:critpts-DbarE}{en:critpt-DbarE:critineqs} implies that
  \begin{equation*}
    \sum_{x\in\Zcal'}v(x)\log\frac{|v(x)|}{\refmeasurex} \le v^{+}(\Zcal')\DbarE(v_{0}) = 0
  \end{equation*}
  for all $v\in\normspace$.  Together with the same inequality with $v$ replaced by $-v$ it follows that
  $\sum_{x\in\Zcal'}v(x)\log\frac{|v(x)|}{\refmeasurex} = 0$.  Hence
  $\sum_{x\in\Zcal}v(x)\log\frac{|v(x)|}{\refmeasurex} = \DbarE(v)-\sum_{x\in\Zcal'}v(x)\log\frac{|v(x)|}{\refmeasurex}
  = 0$.

  If $s\ge 1$, then let $\Ycal=\Xcal\setminus\supp(v_{s})$.
  Let $\Expfam'$ be the exponential family on $\Ycal$ with reference measure the restriction $\refmeasure|_{\Ycal}$ of
  $\refmeasure$ to $\Ycal$ and normal space $\normspace'=\Set{v|_{\Ycal}: v\in\normspace}$.  The case $s=0$ implies
  $\Dbar_{\Expfam'}(w) = \DbarE(v) - \sum_{x\in\supp(v_{s})}v(x)\log\frac{|v(x)|}{\refmeasurex} = 0$ for all
  $w=v|_{\Ycal}\in\normspace'$.  Therefore, the statement follows from induction.
\end{proof}

Let $\ul\Xcal=\Set{x\in\Xcal: v(x)\neq0\text{ for some }v\in\normspace}$.  Define a relation $\sim$ on $\ul\Xcal$ via
\begin{equation*}
  x\sim y \quad\Longleftrightarrow\quad v(y)\neq0\text{ for all }v\in\normspace\text{ such that }v(x)\neq0.
\end{equation*}
It is easy to see that $\sim$ is an equivalence relation: If there exist $v,w\in\normspace$ such that $v(y)\neq 0=v(x)$
and $w(x)\neq0\neq w(y)$, then $u:= v(y)w - w(y)v\in\normspace$ satisfies $u(y)=0\neq u(x)$, and so $\sim$ is symmetric.
Transitivity can be shown similarly.  In the language of matroid theory the equivalence classes are the coparallel
classes.

\begin{lemma}
  \label{lem:equiv-classes-formula}
  A subset $\Zcal\subseteq\Xcal$ is an equivalence class of $\sim$ if and only if there exist circuits
  $\sigma_{0},\sigma_{1},\dots,\sigma_{s}$ of $\normspace$ such that
  \begin{equation*}
    \Zcal=\sigma_{0}\setminus\cup_{j=1}^{s}\sigma_{j},
  \end{equation*}
  and such that $\Zcal\setminus\sigma\in\{\emptyset,\Zcal\}$ for all circuits $\sigma$ of $\normspace$.
\end{lemma}
\begin{proof}
  If $x\not\sim y$ for some $y\in\Xcal$, then there exists a $v\in\normspace$ such that $v(x)\neq 0$ and $v(y)=0$.  By
  Lemma~\ref{lem:circuitinvector} there exists a circuit with the same property.  %, hence $y\notin\Zcal$.
  Conversely, if $y\sim x$, then $y\in\sigma$ for any circuit $\sigma$ such that $x\in\sigma$.
\end{proof}

Let $C\in\Rb^{c\times\Xcal}$ be a matrix such that the rows $c_{1},\dots,c_{c}$ of $C$ form a circuit basis
of~$\normspace$.  Since each circuit basis contains a basis, the rank of $C$ equals the dimension of $\normspace$.
% Denote the rows of $C$ by $c_{1},\dots,c_{c}$.
The columns of $C$ are denoted by $\{C_{x}\}_{x\in\Xcal}$.
\begin{lemma}
  \label{lem:submatrices-rank-one}
  Let $\Zcal$ be an equivalence class of $\sim$.  
  The rank of the submatrix $C|_{\Zcal}$ consisting of those columns $C_{x}$ indexed by $\Zcal$ is one.
\end{lemma}
\begin{proof}
  Let $\Zcal\subseteq\Xcal$.  If the rank of $C|_{\Zcal}$ is larger than one, then there exist two circuit vectors
  $c_{1},c_{2}$ such that $c_{1}|_{\Zcal}$ and $c_{2}|_{\Zcal}$ are linearly independent and have support $\Zcal$.  Let
  $x\in\Zcal$.  Let $v = c_{2}(x)c_{1} - c_{1}(x)c_{2}\in\normspace$.  Then $v|_{\Zcal}\neq 0$ and
  $\supp(v|_{\Zcal})\subseteq\Zcal\setminus\{x\}$.  Therefore, $\Zcal$ is not an equivalence class of $\sim$.
\end{proof}
The main argument of the last proof can be reformulated in terms of the elimination axiom of oriented matroid theory,
cf.~\cite{Oxley92:Matroid_Theory}.  In the language of matroid theory Lemma~\ref{lem:submatrices-rank-one} states that
the coparallel classes of a matroid have corank one.
\begin{proof}[Proof of Theorem~\ref{thm:maxDdE=log2}]
  Suppose $\max\DE=\log(2)$.  By Lemma~\ref{lem:submatrices-rank-one}, the rank of $C$ is bounded from above by the
  number of equivalence classes of $\sim$.  Let $\Zcal$ be an equivalence class of $\sim$.  By definition, the submatrix
  $C|_{\Zcal}\in\Rb^{c\times\Zcal}$ is not the zero matrix.  By Lemmas~\ref{lem:equiv-classes-facial}
  and~\ref{lem:equiv-classes-formula} the rows $c_{i}|_{\Zcal}$ of $C|_{\Zcal}$ satisfy $\sum_{x\in\Zcal}c_{i}(x)=0$.
  Hence each equivalence class must contain at least two elements.  Therefore, the rank of $C$, which equals the
  codimension of $\Expfam$, is bounded from above by $\lfloor\frac N2\rfloor$, and so the dimension of $\Expfam$ is bounded
  from below by $N-1-\lfloor\frac N2\rfloor = \lceil\frac N2\rceil-1$.
\end{proof}
\begin{lemma}
  \label{lem:circuits-and-equiv.classes}
%  Suppose $\max\DE=\log(2)$.
  If the dimension of $\normspace$ equals the number of equivalence classes of $\sim$, then the equivalence classes are
  the circuits of $\normspace$.  In other words, the circuit vectors $c_{1},\dots,c_{c}$ of a circuit basis are in
  bijection with the equivalence classes $\Zcal_{1},\dots,\Zcal_{c}$, such that $\Zcal_{i}=\supp(c_{i})$.  Hence
  $\ol\Expfam$ is the mixture of $\ol\Expfam_{1},\dots,\ol{\Expfam_{c}}$, where $\Expfam_{c}$ is the
  exponential family $\ol\Expfam\cap\spms{\Zcal_{i}}$.
\end{lemma}
\begin{proof}
  Let $\Zcal_{1},\dots,\Zcal_{c'}$ be the set of equivalence classes of $\sim$.  Reorder $\Xcal$ such that the
  equivalence classes are given by consecutive numbers.  Let $\tilde C$ be the matrix obtained from $C$ by doing a Gauss
  elimination through row operations.  By assumption $\tilde C$ has $\dim\normspace=c'$ nonzero rows.  By
  Lemma~\ref{lem:submatrices-rank-one}, the $i$th row $\tilde c_{i}$ of $\tilde C$ has support contained in
  $\Zcal_{i}\cup\dots\cup\Zcal_{c'}$.  In particular, $\supp(\tilde c_{c'})=\Zcal_{c'}$.  Therefore, $\tilde c_{c'}$ is
  a circuit vector.  If $v\in\normspace$ has $v(x)\neq 0$ for some $x\in\Zcal_{c'}$, then $\tilde v = v -
  \frac{v(x)}{\tilde c_{c'}(x)}\tilde c_{c'}(x)$ satisfies $\supp(\tilde v) = \supp(v)\setminus\Zcal_{c'}$.  Hence no
  other circuit intersects $\Zcal_{c'}$.  By induction, $\supp(c_{i})$ equals an equivalence class of $\sim$ for each
  $i$.  The first statement follows from $\supp(c_{i})\neq\supp(c_{j})$ for $1\le i <j\le c$.  The last statement is a
  consequence of Corollary~\ref{cor:disconnected-matroid}.
\end{proof}

\begin{proof}[Proof of Theorem~\ref{thm:maxDdE=log2-partition-case}]
  Assume that the dimension of $\Expfam$ equals $\lceil\frac N2\rceil-1$.  By the proof of Theorem~\ref{thm:maxDdE=log2}
  there must be $m:=\lfloor\frac N2\rfloor$ equivalence classes of $\sim$.  If $N$ is even, then each equivalence
  class has cardinality two.  If $N$ is odd, then there may be one equivalence class $\Zcal$ of cardinality three.  In
  this case, reorder $\Xcal$ such that $\Zcal=\{N-2,N-1,N\}$.
  By Lemma~\ref{lem:circuits-and-equiv.classes} there exists a circuit vector $c\in\normspace$ such that
  $\supp(c)=\Zcal$.  Assume without loss of generality that $c_{N-2}$ and $c_{N-1}$ are positive and that
  $c_{N}=-(c_{N-1}+c_{N-2}) = -1$.  Then
  \begin{equation*}
    \sum_{i=N-2}^{N}c_{i}\log|c_{i}| = -h(c_{N-1},c_{N-2})\neq 0,
  \end{equation*}
  where $h(p,q)$ is the entropy of a binary random variable with probabilities $p,q$.  Therefore, if $N$ is even or if
  $\unimeasure$ is a reference measure of $\Expfam$, then all equivalence classes of $\sim$ have cardinality two.

  By Lemma~\ref{lem:circuits-and-equiv.classes} there are exponential families $\Expfam_{1},\dots,\Expfam_{c}$ such that
  $\Expfam_{i}\subseteq\spms{\Zcal_{i}}$ for $i=1,\dots,c$ and such that $\ol\Expfam$ is the mixture of
  $\ol{\Expfam_{1}},\dots,\ol{\Expfam_{c}}$.  For $i=1,\dots,c$ there is a unique circuit vector with support
  $\Zcal_{i}$, hence $\Expfam_{i}\neq\spms{\Zcal_{i}}$, so $\Expfam_{i}$ has dimension $|\Zcal_{i}|-1$.  If
  $|\Zcal_{i}|=2$, then $\Expfam_{i}$ consists of the uniform distribution $\frac12\unimeasure_{\Zcal_{i}}$ on
  $\Zcal_{i}$, so $\ol{\Expfam_{i}}$ is a partition model, and also the mixture of $\ol{\Expfam_{i}}$ for those $i$
  satisfying $|\Zcal_{i}|=2$ is a partition model.
\end{proof}

\begin{proof}[Proof of Proposition~\ref{prop:maxDdE=log2,N=3}]
  Let $\Expfam$ be a one-dimensional exponential family with normal space $\Rb u$.  Without loss of generality assume
  that $u^{+}$ and $u^{-}$ are probability measures.  By Theorem~\ref{thm:Dualmaxi-glob} the set of local maximizers of
  $\DE$ consists of $u^{+}$ and $u^{-}$, and both are projection points.  $\Expfam$ satisfies $\max\DE=\log(2)$ if
  and only if $\projE{(u^{+})}=\projE{(u^{-})}=\frac12(u^{+}+u^{-})$, which happens if and only if $u^{+}+u^{-}$ is a
  reference measure of $\Expfam$, proving existence and uniqueness of~$\Expfam$.
\end{proof}

\section{Optimal exponential families}
\label{sec:opti}

Corollary~\ref{cor:max-ge-log2} says that $\max\DE\ge\log(2)$ for all exponential families $\Expfam\neq\spmsX$.
Therefore $D$-optimality is only interesting for $D\ge\log(2)$.  The case $D=\log(2)$ was studied in
Section~\ref{sec:maxDdE=log2}, where it was shown that $\DNk=\log(2)$ if and only if $\lceil \frac N2\rceil-1\le k< N$.
This condition is equivalent to $\lceil\frac{N}{k+1}\rceil=2$.  Many $\log(2)$-dimension optimal exponential families
are partition exponential families.

\begin{ex}
  \label{ex:0D-optimality}
  Any zero-dimensional exponential family $\Expfam=\{\refmeasure\}$ is dimension-optimal.
  The function $P\mapsto \IDr{P}$ is convex on the probability simplex $\pmsX$ and
  attains its maximum at a vertex of $\pmsX$, which corresponds to a point distribution.  Therefore,
  \begin{equation*}
    \max\DE=\max\{-\log(\refmeasure_{x}):x\in\Xcal\}\ge\log|\Xcal|.
  \end{equation*}
  Hence $\DNkof{N}{1}=\log(N)$, and $\Expfam$ is $D$-optimal if and only if $\refmeasure_{x}\ge e^{-D}$ for all
  $x\in\Xcal$.  Zero-dimensional exponential families are the dimension $D$-optimal exponential families for
  $D\ge\log|\Xcal|$.  In general, they are not the only inclusion $D$-optimal exponential families, see
  Example~\ref{ex:3-optimality}.
\end{ex}

\begin{ex}
  \label{ex:3-optimality}
  Let $\Xcal=\{1,2,3\}$.  Any zero-dimensional exponential family $\Expfam=\{\refmeasure\}$ satisfies
  $\max\DE% =\max\{-\log(\refmeasure_{1}),-\log(\refmeasure_{2}),-\log(\refmeasure_{3})\}
  \ge\log(3)$.  Therefore, if $\log(2)\le D<\log(3)$, then the dimension $D$-optimal exponential families are
  one-dimensional.
  The normal space $\normspace$ of any one-dimensional exponential family $\Expfam$ is spanned by a single element $u$,
  which can be taken to be normalized, such that $\bdKtopA = \{\pm u\}$.  By Theorem~\ref{thm:Dualmaxi-glob} the set of
  local maximizers of $\DE$ equals $\{u^{+},u^{-}\}$.  Let $\projPE=\projE{(u^{+})}=\projE{(u^{-})}$, then $\projPE=\mu
  u^{+} + (1-\mu)u^{-}$ for some $0<\mu<1$.  Hence $\DE(u^{+}) = -\log\mu$ and $\DE(u^{-}) = -\log(1-\mu)$.  It follows
  that $\Expfam$ is dimension $D$-optimal if and only $e^{-D}\le\mu\le 1 -e^{-D}$.  Alternatively, using
  Theorem~\ref{thm:Dualmaxi-glob}, $\Expfam$ is dimension $D$-optimal if and only if $-\log(e^{D}-1) \le \DbarE(u)\le
  \log(e^{D}-1)$.

  If $D\ge\log(3)$, then the dimension $D$-optimal exponential families are zero-di\-men\-sion\-al, consisting of a
  single point $\{\refmeasure\}$ such that $\min\{\refmeasure_{1},\refmeasure_{2},\refmeasure_{3}\}\ge e^{-D}$.  There
  are also one-dimensional inclusion $D$-optimal exponential families: Consider, for example, the exponential family
  $\Expfam$ with sufficient statistics $\suffstat=(0,1,2)$
  and reference measure $\refmeasure=(1,4,1)$.  The two local maximizers are $u^{+}=\delta_ 2$ and $u^{-}=\frac12(\delta_ 1 +
  \delta_ 3)$.  Their $rI$-projection is $\projPE=\frac16\refmeasure$.  Hence $\DE(u^{+}) = \log\frac32$ and $\DE(u^{-}) =
  \log3$, and so $\max\DE=\log3$.  The monomial parametrization of $\Expfam$ is
  \begin{equation*}
    P_{\xi} = \frac{1}{Z_{\xi}}(1,4\xi,\xi^{2}),
  \end{equation*}
  where $\xi\in\Rnn$ and $Z_{\xi}=1+4\xi+\xi^{2}$.  Consequently, $\Expfam$ does not contain the uniform distribution.
  Therefore, any point $P\in\Expfam$ satisfies $\max\ID{\cdot}{P}>\max\DE$.
\end{ex}

The following theorem generalizes the special case of Theorem~\ref{thm:maxDdE=log2-partition-case} when $N$ is even.

\begin{thm}
  \label{thm:logNk-optimality}
  Let $\Xcal$ be a finite set of cardinality $N$.  Then $\DNk\ge\log(N/(k+1))$ for all $0\le k< N$.  If $\Ecal$
  is a $k$-dimensional exponential family that satisfies
  $\max\DE=\log(N/(k+1))$, then % $k+1$ divides $N$, and
  $\Ecal$ is a partition model of a homogeneous partition of coarseness $N/(k+1)$.  In particular, if $N$ is divisible
  by $(k+1)$, then $\DNkof{N}{k} = \log(N/(k+1))$, and the dimension $\DNkof{N}{k}$-optimal models % among $\expfamuniref$
  are partition models.
\end{thm}
\begin{proof}
  First assume that $\Expfam\in\expfamuniref$.
  Let $\suffstat$ be a sufficient statistics of $\Expfam$.  The moment map $\mommap$ maps the uniform distribution
  $Q=\frac1N\unimeasure$ to a point in the relative interior of $\csA$.  By Carathéodory's theorem
  there are $k+1$ vertices $\suffstat_{x_{0}},\dots,\suffstat_{x_{k}}$ of $\csA$ and
  $\lambda_{0},\dots,\lambda_{k}\in\Rnn$ such that $\mommap(Q)=\sum_{i=0}^{k}\lambda_{i}\suffstat_{x_{i}}$ and
  $\sum_{i=0}^{k}\lambda_{i}=1$.  Let $P=\sum_{i=0}^{k}\lambda_{i}\delta_{x_{i}}$, then $Q = \projPE$.  By the
  Pythagorean theorem, $\max\DE\ge\DE(P) = H(Q)-H(P)\ge\log(N)-\log(k+1)$, proving the first assertion.
  
  If equality holds, then $\lambda_{0}=\dots=\lambda_{k}=\frac1{k+1}$.  Let
  $x\in\Xcal\setminus\{x_{0},\dots,x_{k}\}$.  For $i\in\{0,\dots,k\}$ let $C_{i}$ be the convex hull of
  $\suffstat_{x_{0}},\dots,\suffstat_{x_{i-1}},\suffstat_{x_{i-1}},\dots,\suffstat_{x_{k}}$ and $\suffstat_{x}$.
  By Carathéodory's theorem the sets $C_{i}$ cover the convex hull of $\suffstat_{x_{0}},\dots,\suffstat_{x_{k}}$
  and $\suffstat_{x}$.  In particular, $\mommap(Q)\in C_{j}$ for some $j\in\{0,\dots,k\}$, so $\mommap(Q)=\sum_{i\neq
    j}\lambda'_{i}\suffstat_{x_{i}} + \lambda'_{j}\suffstat_{x}$.  By the same argument as above it follows that
  $\lambda'_{0}=\dots=\lambda'_{k}=\frac{1}{k+1}$.  Therefore, $\suffstat_{x}=(k+1)\mommap(Q)-\sum_{i\neq
    j}\suffstat_{x_{i}}=\suffstat_{x_{j}}$.

  Let $\sim$ be the equivalence relation on $\Xcal$ defined by $x\sim y$ if and only if
  $\suffstat_{x}=\suffstat_{y}$, and let $\Xcal'=(\Xcal^{1},\dots,\Xcal^{N'})$ be the corresponding
  partition into equivalence classes.  Then $N'\le k+1$ by what was shown until now.  From $\dim(\Expfam)=\dim(\csA)$
  one concludes $N'=k+1$, and $\csA$ is a simplex of dimension $k$.  By Lemma~\ref{lem:cs-simplex-partmodel}, $\Expfam$
  equals the partition model of $\Xcal'$.  Lemma~\ref{lem:partition-maxDdE} implies that the coarseness of
  $\Xcal'$ equals $\frac N{k+1}$, which must be an integer.  Furthermore, $\Xcal'$ is homogeneous.

  It remains to prove $\max\DE>\log(N/(k+1))$ % the inequality
  in the case $\Expfam\notin\expfamuniref$.
  Let $\projPE$ be the $rI$-projection of the uniform distribution, and let $\Ncal_{\unimeasure}$ be the set of
  probability distributions that $rI$-project to $\projPE$.  The function $\DE$ is convex on $\Ncal_\unimeasure$, hence
  $\DE$ is maximal at the vertices of $\Ncal_\unimeasure$.  Let $P$ be a vertex of $\Ncal_{\unimeasure}$.  Assume that
  $v\in\Ncal$ satisfies $\supp(v)\subseteq\supp(P)$.  Then there exists $\epsilon>0$ such that $P\pm\epsilon
  v\in\Ncal_{\unimeasure}$ and $P = \frac12(P + \epsilon v) + \frac12(P - \epsilon v)$.  Hence $v=0$.  Therefore, the
  set $\{A_{x}:P(x)>0\}$ is linearly independent.  In particular $|\supp(P)|\le\dim(\Xcal)+1$.

  Denote by $\Expfam_{\unimeasure}$ the exponential family with uniform reference measure and with the same normal space
  as $\Expfam$.  On $\Ncal_\unimeasure$ the difference
  \begin{equation*}
    \delta(P) := \DE(P) - D_{\Expfam_{\unimeasure}}(P)
    = - \sum_{x\in\Xcal}P(x)\log\projPE(x) - \log N
  \end{equation*}
  is an affine function that is positive at the uniform distribution.  Hence there is a vertex $P$ of
  $\Ncal_\unimeasure$ such that $\delta(P)>0$, and so
  $\DE(P)> D_{\Expfam_{\unimeasure}}(P)=\log N-H(P)\ge\log(N/(k+1))$.
\end{proof}

The value of $\DNk$ is unknown when $k+1$ does not divide $N$.
The situation is known for $N=3$, see Example~\ref{ex:3-optimality}: If $1\le k<3$, then $\DNkof{N}{k}=\log(2)$, and
all dimension $\DNkof{N}{1}$-optimal exponential families that contain the uniform distribution are partition models.  The following conjecture generalizes this example and Theorems~\ref{thm:maxDdE=log2-partition-case} and~\ref{thm:logNk-optimality}:
\begin{conj}
  \label{con:DNk-partmod}
  $\DNk = \log\lceil\frac{N}{k+1}\rceil$, and the dimension $\DNk$-optimal exponential families containing the uniform
  distribution are partition models.
\end{conj}
The following weaker statement holds:
\begin{lemma}
  \label{lem:inclusion-optimal-partition-models}
  Let $\Xcal'=\left\{\Xcal^{1},\dots,\Xcal^{N'}\right\}$ be a partition of coarseness $c<N$ such that
  $\Xcal^{1}$ has cardinality $l\le c$ and all other components $\Xcal^{i}$ for $i>1$ have cardinality $c$.
  Then the partition model $\Expfam$ of $\Xcal'$ is $\log(c)$-inclusion optimal.
\end{lemma}
\begin{proof}
  The fact that $\max\DE = \log(c)$ follows from Lemma~\ref{lem:partition-maxDdE}.  It remains to prove the optimality.
  Let $\Expfam'\subseteq\Expfam$ be an exponential family contained in $\Expfam$.  Let $\Zcal$ be the union of all
  blocks of $\Xcal'$ of cardinality $c$.  Assume that there exists a probability measure
  $Q\in\ol\Expfam\setminus\ol{\Expfam'}$ with support contained in $\Zcal$.  By
  Corollary~\ref{cor:part-mod:any-Q-is-rI-of-glob-max} there exists $P\in\pms(\Zcal)$ such that $Q=\projPE$ and
  $D(P\|Q)=\log(c)$.  Let $Q'=\proj{P}{\Expfam'}\in\Expfam$.  Then $D(P\|Q') = D(P\|Q) + D(Q\|Q')>\log(c)$ by the
  Pythagorean identity.
  Otherwise, if $\ol\Expfam\cap\pms(\Zcal)=\ol{\Expfam'}\cap\pms(\Zcal)$, then
  $\dim(\Expfam) = \dim(\ol\Expfam\cap\pms(\Ycal)) + 1=\dim(\ol\Expfam'\cap\pms(\Ycal)) + 1\le\dim(\Expfam')$, so
  $\Expfam=\Expfam'$.
\end{proof}

Theorem~\ref{thm:logNk-optimality} can be applied to the hierarchical models $\Expfam_{K}$ for $K\subseteq\cardn$
introduced in Remark~\ref{rem:hier-partmod}.
By Theorem~\ref{thm:logNk-optimality} the hierarchical model $\Expfam_{K}$ is dimension optimal with $\max
D(\cdot\|\Expfam_{K}) = \sum_{i\in\cardn\setminus K}\log(N_{i})$.
% among exponential families with uniform reference measures.
If $N_{\numvar}=2$, then the choice $K = \{1,\dots,\numvar-1\}$ yields an exponential family of dimension less than
$|\Xcal|/2$ such that $\max D(\cdot\|\Expfam_{K}) = \log(2)$, and Theorem~\ref{thm:maxDdE=log2} implies that
$\Expfam_{K}$ is dimension optimal.
The following proposition says that the exponential families $\Expfam_{K}$ are the unique dimension $D$-optimal
hierarchical models % among exponential families with uniform reference measures
for many values of $D$.
\begin{prop}
  Let $\Xcal=\Xcal_{1}\times\dots\times\Xcal_{\numvar}$, where $N_{i}=|\Xcal_{i}|<\infty$.  For any
  $K\subseteq\cardn$ let $D_{K}=\sum_{i\notin K}\log(N_{i})$.  The hierarchical model $\Expfam_{K}$ is dimension
  $D_{K}$-optimal. % among $\expfamuniref$.
% all exponential families with uniform reference measure.

  Let $l$ be any divisor of $N:=|\Xcal|=\prod_{i=1}^{\numvar}N_{i}$.  If $\Expfam$ is any hierarchical model
  that is dimension $\log(N/l)$-optimal, % among $\expfamuniref$,
  then there is a subset $K\subseteq\cardn$ such that $\Expfam=\Expfam_{K}$.
\end{prop}
The proposition implies that if $l$ is not of the form $\prod_{i\in K}N_{i}$ for some subset $K\subseteq\cardn$, then
there exists no hierarchical model that is dimension $\log(N/l)$-optimal.  % among $\expfamuniref$.
\begin{proof}
  It only remains to prove the last statement.  If $\Expfam$ satisfies the assumptions, then $\Expfam$ is a partition
  model by Theorem~\ref{thm:logNk-optimality}.  Therefore, it suffices to prove that any hierarchical model that is also
  a partition model is of the form $\ol{\Expfam_{K}}$.

  Let $\Delta$ be a simplicial complex on $\cardn$ such that $\Expfam=\Expfam_{\Delta}$, and let $K=\cup_{J\in\Delta}J$.
  Then $\Expfam$ is a submodel of $\Expfam_{K}$.  Let $\suffstat$ be a sufficient statistics of $\Expfam$.  By
  Lemma~\ref{lem:csA-of-hierarchical} the convex supports of $\Expfam$ and $\Expfam_{K}$ have the same number of
  vertices.  By Lemma~\ref{lem:cs-simplex-partmodel} both are simplices, hence they have the same dimension, so
  $\Expfam=\Expfam_{K}$.
\end{proof}

\section{Discussion}
\label{sec:conclusion}

Conjecture~\ref{con:DNk-partmod} would imply that the partition models of
Lemma~\ref{lem:inclusion-optimal-partition-models} are dimension optimal among all exponential families.  If the
conjecture were true, then it would suggest the following interpretation: In many cases the information divergence $\ID
PQ$ can be interpreted as the information which is lost when $P$ is the true probability distribution, but computations
are carried out with $Q$.  For example, in the case of the independence model $\Expfam_{1}$ of two variables,
$D_{\Expfam_{1}}$ equals the mutual information and measures the amount of information that
one variable carries about the other variable.  If a probability measure is replaced by its $rI$-projection, then this
information is lost.

For the exponential families $\Expfam_{K}$ the loss equals $D_{K}=\sum_{i\notin K}\log(N_{i})$, which is precisely the
maximal information that the random variables that are not in $K$ can carry.  Assuming that the conjecture is true, if
the model is smaller than $\Expfam_{K}$, then, in general, more information can be lost.  In this interpretation the
fact that $\max\DE\ge\log(2)$ unless $\Expfam=\spmsX$ means that for any exponential family $\Expfam\neq\spmsX$ in
general at least one bit is necessary to compensate the approximation of arbitrary probability measures.

\subsubsection*{Acknowledgements}

I thank Yaroslav Boulatov, who first asked the main question studied in this work.  Further thanks goes to Nihat Ay,
whose questions led me in similar directions, but from a different starting point.

\bibliographystyle{IEEEtranSpers}
\bibliography{general}

\end{document}